\RequirePackage[l2tabu, orthodox]{nag}
\documentclass[a4paper, 12pt,reqno]{amsart}

\usepackage[hcentering, width=456pt, lines=46]{geometry}

\usepackage[utf8]{inputenc}
\usepackage{microtype}
\usepackage{amsmath}
\usepackage{amsthm}
\usepackage{amsfonts}
\usepackage{amssymb}
\usepackage{mathtools}
\usepackage{tikz-cd}

\usepackage{enumitem}

\usepackage{hyperref}
\hypersetup{
	colorlinks,
	citecolor=blue,
	filecolor=green,
	linkcolor=darkgray,
	urlcolor=gray
}

\author{Boris Bilich}
\address{Department of Mathematics, HSE University, Usacheva str. 6, 119048, Moscow, Russian Federation}
\email{bilichboris1999@gmail.com}
\title{Taylor spectrum for modules over Lie algebras}

\subjclass[2020]{Primary 17B56, Secondary 17B30, 47A13}

\keywords{Taylor spectrum, Lie algebra cohomology}

\theoremstyle{plain}      
\newtheorem{theorem}{Theorem}[section]  
\newtheorem{corollary}[theorem]{Corollary}     
\newtheorem{lemma}[theorem]{Lemma}     
\newtheorem{proposition}[theorem]{Proposition}

\theoremstyle{remark}      
\newtheorem{example}[theorem]{Example} 
\newtheorem*{remark}{Remark} 

\newtheorem*{acknowledgments}{Acknowledgments}
\theoremstyle{definition}
\newtheorem{definition}[theorem]{Definition}

\newcommand{\cmpl}{\mathbb{C}}
\newcommand{\lieg}{\mathfrak{g}}
\newcommand{\lieh}{\mathfrak{h}}
\newcommand{\lies}{\mathfrak{s}}
\newcommand{\lien}{\mathfrak{n}}
\newcommand{\Ext}{\operatorname{Ext}}
\newcommand{\Hom}{\operatorname{Hom}}
\newcommand{\ad}{\operatorname{ad}}
\newcommand{\Tor}{\mathrm{Tor}}
\newcommand{\ZZ}{\mathbb{Z}}

\newcommand{\CC}{\mathbb{C}}

\newcommand{\lmod}[1]{\operatorname{\mathbf{#1-mod}}}
\newcommand{\rmod}[1]{\operatorname{\mathbf{mod-#1}}}

\begin{document}

\begin{abstract}
    In this paper we generalize the notion of the Taylor spectrum to modules over an arbitrary Lie
    algebra and study it for finite-dimensional modules. We show that the spectrum can be described as the set of simple submodules in case of nilpotent and semisimple Lie algebras. We also
    show that this result does not hold for solvable Lie algebras and obtain a precise description
    of the spectrum in case of Borel subalgebras of semisimple Lie algebras.
\end{abstract}

\maketitle

%\tableofcontents
\section{Introduction}%
\label{sec:introduction}
In 1970, Taylor introduced the notion of the joint spectrum for a finite tuple
of commuting operators $T = (T_1,...,T_n)$ on a Banach space $V$ in \cite{Taylor1970}. It was defined as the
set of points
$\lambda=(\lambda_1,...,\lambda_n)\in \CC^n$ such that a certain Koszul complex associated to $\lambda$ is not exact. The
Taylor spectrum of a single operator coincides with the classical notion of spectrum. The same year, Taylor estabilished the
existence of the holomorphic functional calculus in a neighborhood of the spectrum in 
\cite{Taylor1970b}. In 1972, he proposed a framework for a noncommutative functional
calculus in \cite{Taylor1972} but the notion of spectrum for non-commuting tuples of operators was not
yet developed. 

The first step in this direction was made by Fainshtein in \cite{Fainshtein}. He
generalized the Taylor spectrum to tuples of operators generating a finite dimensional nilpotent Lie algebra.
In \cite{boasso}, Boasso and Larotonda independently introduced the Taylor spectrum for representations of solvable Lie algebras. It is given by \begin{equation} \label{e:boassolarotonda}
	\sigma^{BL}_\lieg(V) = \{ \lambda \in (\lieg/[\lieg,\lieg])^*\colon \Tor^{U\lieg}_k(V,
	\CC_\lambda)\neq 0 \text{ for some } k \in \ZZ_{\geq0}\}
\end{equation}
where $\lieg$ is a solvable Lie algebra and $V$ is a right Banach $\lieg$-module.

The spectrum $\sigma^{BL}$ is equivalent to the definition of Fainshtein in the nilpotent case. Moreover, we can think about commuting tuples $T=(T_1,...,T_n)$ as representations of the $n$-dimensional abelian Lie algebra and it turns out that the spectrum for abelian Lie algebras is equivalent to the original Taylor spectrum. 

The Taylor spectrum for Lie algebras was actively studied in works of several authors. A variety of classical theorems were generalized, such as the projection property and different kinds of the spectral mapping theorem. The details can be found in the monograph \cite{beltita} and in a series of works by Dosi (see for example
\cite{Dosi1,Dosi2,Dosi3}). In 2010, Dosi estabilished a noncommutative holomorphic functional calculus in a neighborhood of the spectrum for supernilpotent Lie algebras of operators \cite{Dosi4}.

In this paper we introduce the notion of Taylor spectrum for modules over an arbitrary finite-dimensional Lie algebra and study it for finite-dimensional modules. Let $\hat\lieg$ be the set of isomorphism classes of simple $\lieg$-modules. The spectrum of a $\lieg$-module is given as a subset of $\hat\lieg$ by the non-vanishing property of some homology groups (Definition \ref{d:tspec}).  

The article is organized as follows. In Section \ref{sec:preliminaries} we recall neccesary facts about Lie algebra cohomology. In Section \ref{sec:spectrumofmodule} we give the definition of the spectrum and study its properties. We obtain a complete description for the spectrum of finite-dimensional modules over a semisimple Lie algebra in Theorem \ref{t:ssspectrumissubmodules}. In this case, the spectrum coincides with the set of isomorphism classes of simple submodules. This result holds even for Banach modules as shown in Corollary \ref{cor:banachss}. We finish the section with Proposition \ref{p:1ext}, which shows the behaviour of the spectrum for Lie algebra extensions.

Section \ref{sec:solvable} is devoted to the spectrum for solvable Lie algebras. Although this case was studied before (see \cite{beltita,boasso}), the authors considered mostly infinite-dimensional modules. We explore the spectrum for finite-dimensional modules. It turns out that the spectrum is hard to compute even for the trivial module. An approximation is given by Theorem \ref{t:jordanholder}, which happens to be a refolmutaion in terms of the Taylor spectrum of an earlier result by Millionschikov \cite{mill}. We describe completely the spectrum in case of nilpotent Lie algebras (Theorem \ref{t:nilpspectrum}) and Borel subalgebras of semisimple Lie algebras (Theorem \ref{t:borelspectrum}).

\begin{acknowledgments}
	The author is grateful to his scientific advisor Alexei Pirkovskii for his careful guidance during writing this paper. Thanks should also go to Egor Alimpiev and Marat Rovinsky for language corrections to an early draft and to Simon Wadsley and Yves de Cornulier for useful discussions on the Mathoverflow website (see answers to \cite{MO}).
\end{acknowledgments}

\section{Preliminaries}%
\label{sec:preliminaries}
\subsection{Notation} 
In this article all algebras, including Lie algebras, are defined over the field of complex numbers. We use the notation $U\lieg$ for the enveloping algebra of a Lie algebra $\lieg$. We denote by $\lmod{\lieg}$ and
$\rmod\lieg$ the categories of left and right $\lieg$-modules, respectively.  We write $\hat\lieg$
for the set of isomorphism classes of simple finite-dimensional $\lieg$-modules and $\CC$
for the trivial module. If $S$ is a simple $\lieg$-module we also denote by $S\in \hat\lieg$ its isomorphism class. This will not lead to a confusion.

\subsection{Functors  on categories of modules}
For the rest of this section we denote by $\lieg$ an arbitrary finite-dimensional Lie algebra.
 Let $V$ be a $\lieg$-module. We define vector spaces 
\[
	V^\lieg = \{v\in V\colon g\cdot v = 0 ~ \forall g\in \lieg \}, 
\]
called \textit{invariants} and
\[
	V_\lieg = V/\lieg V,
\]
called \textit{coinvariants}. The assignments $V \mapsto V^\lieg$ and $V\mapsto V_\lieg$ define functors from $\lmod\lieg$ (or $\rmod\lieg$) to the category of vector spaces. These functors are isomorphic to $\Hom_{U\lieg}(\CC, V)$ and $\CC \otimes_{U\lieg} V$ respectively.

We define two functors $\square^*\colon \lmod{\lieg}^{op} \to \lmod{\lieg}$ and
$\square^\circ\colon \lmod{\lieg} \to \rmod{\lieg}$ as follows. The \textit{duality functor} $\square^*$ maps a $\lieg$-module V to its dual vector space with the left $\lieg$-action given by
\[
(g\cdot f) (v) = -f(g\cdot v),~\text{ for all } f\in V^*,~v \in V,~g\in\lieg.
\]
The \textit{antipode functor} $\square^\circ$ maps $V$ to itself as a vector space with the right $\lieg$-action given by
\[
v\cdot g = - g\cdot v,~\text{ for all } v \in V,~g\in\lieg.
\]
The antipode functor defines an equivalence of the categories $\lmod\lieg$ and $\rmod\lieg$. We also denote by $\square^*$ and
$\square^\circ$ the functors on the category of right $\lieg$-modules, defined in the same way. It is easy to see that $(\square^\circ)^\circ$ is naturally
isomorphic to the identity functor. However, $(\square^*)^*$ is naturally isomorphic to the identity functor only on the subcategory of finite-dimensional modules.

Another important functor is $\square \otimes_\CC \square \colon \rmod\lieg \times
\lmod\lieg \to \lmod\lieg$. If $V \in \rmod\lieg$ and $W \in \lmod\lieg$, then $V \otimes_\CC W$
is the tensor product of $V$ and $W$ as vector space with the action of $\lieg$, which is determined by
the formula
\[
g\cdot (v\otimes w) = v\otimes (g\cdot w) - (v\cdot g) \otimes w, \text{ for all } w\in W,~v\in
V,~ g\in \lieg.
\]
For $V, W \in \lmod\lieg$ (resp. $\rmod\lieg$), we denote by $V\otimes W$ the left (resp. right) $\lieg$-module
$V^\circ \otimes_\CC W$ (resp. $V \otimes_\CC W^\circ$).

For $V\in \lmod\lieg$ and and a simple $\lieg$-module $S \in \hat\lieg$, we write $V_{S}$ for the $\lieg$-module
$S\otimes_\CC V$.  If $S$ is one-dimensional, then it can be described by a character $\lambda \in
(\lieg/[\lieg, \lieg])^*$. In this case we write $V_\lambda$ for $V_S$. For example, let $\CC$ be the trivial $\lieg$-module. 
With use of the above notation $\CC_\lambda$ stands for the one-dimensional module with the action given by $g\cdot s = \lambda(g)s$ for
all $s\in \CC_\lambda$ and $g \in \lieg$. We also use the notation $V_{-S}$ for the module
$S^* \otimes_\CC V$. If $S$ is one-dimensional of character $\lambda$, then $V_{-S}$ is isomorphic to $V_{-\lambda}$. 

\subsection{Lie algebra (co)homology}
In this subsection we briefly recall necessary definitions and properties of Lie algebra cohomology. For details, we refer the reader to any related textbook (see for example \cite[Chapter 7]{weibel}).

\begin{definition}
	Let $V$ be a left $\lieg$-module. For all $k\in \ZZ_{\geq0}$, the $k$-th \textit{homology} of $\lieg$ with coefficients in $V$ is defined as 
	\[
		H_k(\lieg, V) = \Tor_k^{U\lieg}(\cmpl, V).
	\]
	Dually, the $k$-th \textit{cohomology} is defined as
	\[
		H^k(\lieg, V) = \Ext^k_{U\lieg}(\CC, V).
	\]
\end{definition}

(Co)homologies are functors from the category $\lmod\lieg$ to the category of vector spaces. They can be computed with the Chevalley-Eilenberg free resolution of the trivial $\lieg$ module $\cmpl$ (see \cite[Definition 7.7.1]{weibel}). It has $F_k=U\lieg\otimes_\CC \bigwedge^k
\lieg$ in degree $k$ with the differential given by
\[
	d(u\otimes g_1 \wedge \cdots \wedge g_p) 
	= \sum_{i = 1}^p (-1)^{i+1} u g_i \otimes g_1 \wedge \cdots \wedge \hat{g}_i\wedge \cdots \wedge g_p  
	+$$ $$+\sum_{i < j} (-1)^{i+j} u\otimes [g_i, g_j] \wedge \cdots \wedge \hat{g}_i\cdots \wedge
	\hat{g}_j\cdots \wedge g_p,~\text{where } u\in U\lieg,~g_i \in \lieg.
\]
We need the following fact about homology.
\begin{lemma} \label{t:torhomology}
	Let $V \in \rmod\lieg$ and $W \in \lmod\lieg$. Then  
	\[
	\Tor_k^{U\lieg}(V,W) \cong \Tor_k^{U\lieg}(\cmpl,V\otimes_\cmpl W) = H_k(\lieg, V\otimes_\cmpl W),
	\]
	for all $k\in \ZZ_{\geq 0}$.
\end{lemma}
\begin{proof}
	Since the functors $\CC\otimes_{U\lieg}(\square \otimes_\CC W)$ and $\square \otimes_{U\lieg} W$
	are naturally isomorphic, it suffices to show that if $P_\bullet \to V$ is a flat resolution
	of $V$, then $P_\bullet \otimes_\CC W$ is a flat resolution of $V\otimes_\CC W$.
	
	By definition, the flatness of $P_k$ means the exactness of the functor $P_k\otimes_{U\lieg}\square$. We use the
	associativity of the tensor product to obtain a natural isomorphism of functors $(P_k\otimes_\CC W)
	\otimes_{U\lieg} \square$ and $P_k\otimes_{U\lieg} (W \otimes_\CC \square)$. The last functor
	is the composition of exact functors, so it is also exact. If follows that $P_k\otimes_\CC W$ is flat and we obtain the  natural isomorphism $\Tor_k^{U\lieg}(V,W) \cong \Tor_k^{U\lieg}(\cmpl,V\otimes_\cmpl W)$.
\end{proof}

A useful variation of the Poincar\'e duality holds for finite dimensional Lie algebras. Let $n$ be the dimension of $\lieg$. We endow the one-dimensional vector space $\bigwedge^n \lieg$ with the structure of a left
$\lieg$-module, which extends the adjoint action by the Leibnitz rule. That is, 
\[
g\cdot g_1\wedge \cdots \wedge g_n =\sum_{i=1}^n g_1\wedge \cdots \wedge[g, g_i]\wedge \cdots \wedge g_n
\]
for $g, g_i \in \lieg$.

\begin{proposition}[Poincar\'e duality] \label{t:poincare}
	For $0 \leq k \leq n$ and a left $\lieg$-module $V$, there are vector space isomorphisms 
	\[
	H^k(\lieg, V^*) \cong H_k(\lieg, V)^*~ \text{ and }~
	H^k(\lieg, V) \cong H_{n - k}(\lieg, (\bigwedge^n \lieg)^*\otimes_\cmpl V),
	\]
	natural in $V$. Consequently
	\[
	H^k(\lieg, V^*) \cong H^{n - k}(\lieg,  \bigwedge^n \lieg\otimes_\cmpl V)^*. 
	\]
\end{proposition}
\begin{proof}
	\cite[Theorem 6.10]{knapp}.
\end{proof}

\section{Taylor spectrum of \texorpdfstring{$\lieg$-modules}{g-modules}}%
\label{sec:spectrumofmodule}

\subsection{Definition and first properties}
Let $\lieg$ be an arbitrary Lie algebra of dimension $n$ and $V$ be a left $\lieg$-module.  Recall that $\hat\lieg$
is the set of isomorphism classes of simple finite-dimensional $\lieg$-modules.
\begin{definition} \label{d:tspec}
	The \textit{Taylor spectrum} of $V$ is the subset of $\hat\lieg$ defined as
	\[
	\sigma_\lieg(V)=\{ S\in \hat{\lieg} \colon 
	\Tor^{U\lieg}_k(S^*,V)\neq 0 \text{ for some } k\geq 0 \}.
	\]
\end{definition}
We will occasionaly use the notation $\sigma(V)$ instead of $\sigma_\lieg(V)$ when it is clear what $\lieg$ is meant. We will also write \textit{the spectrum} for \textit{the Taylor spectrum}. 

Suppose that $\lieg$ is a solvable Lie algebra. By Lie's theorem, every simple $\lieg$-module is one-dimensional. Thus, $\hat\lieg$ can be identified with the space $(\lieg/[\lieg, \lieg])^*$ of characters. In this case, the Taylor spectrum can be expressed in terms of the spectrum \eqref{e:boassolarotonda}, defined by Boasso and Larotonda in \cite{boasso}:
\[
\sigma_\lieg(V) = -\sigma^{BL}_\lieg(V^\circ).
\]
We may think about $\sigma_\lieg$ as a generilization of $\sigma^{BL}_\lieg$. 

The spectrum has other useful characterizations.
\begin{proposition} \label{p:definitions}
	For an arbitrary $S\in\lieg$, the following are equivalent:
	\begin{enumerate}
		\item $S \in \sigma(V)$; 
		
		\item $H_k(\lieg, V_{-S}) \ne 0$ for some $k$; 
		
		\item $H^k(\lieg, \bigwedge^n\lieg\otimes_\CC V_{-S}) \ne 0$ for some $k$;
		
		\item $S^* \otimes_\CC \bigwedge^n\lieg \in \sigma(V^*)$.
	\end{enumerate}
\end{proposition}
\begin{proof}
	\begin{itemize}
		\item[$1 \Leftrightarrow 2$]
		This follows immediatley from Lemma \ref{t:torhomology}.
		
		\item[$2 \Leftrightarrow 3$]
		This is just the Poincar\'e duality (Proposition \ref{t:poincare}).
		
		\item[$3 \Leftrightarrow 4$]
		Again, by the Poincar\'e duality we have
		\[
		H^k(\lieg, \bigwedge^n\lieg \otimes_\CC V_{-S})^* = H^{n - k} (\lieg, (V_{-S})^*).
		\]
		Observe that $(V_{-S})^*$ is isomorphic to $S\otimes_\CC V^*\cong \bigwedge^n\lieg\otimes_\CC (\bigwedge^n\lieg)^* \otimes_\CC S\otimes V^*\cong \bigwedge^n\lieg \otimes_{\CC}
		(V^*)_{-S^* \otimes_\CC \bigwedge^n\lieg}$. Using the equivalence of (1) and (3) we obtain the desired result. 
	\end{itemize}
\end{proof}

\begin{lemma} \label{l:exactsequence}
	Let $V$, $V'$, and $V''$ be $\lieg$-modules such that there is a short exact sequence
	\[
	0\to V' \to V \to V'' \to 0.
	\]
	Then
	\begin{enumerate}
		\item  $\sigma_\lieg(V)\subset \sigma_\lieg(V')\cup \sigma_\lieg(V'')$.
		
		\item If the short exact sequence splits, i.e., $V\cong V'\oplus V''$ then $\sigma_\lieg(V)=\sigma_\lieg(V')\cup \sigma_\lieg(V'')$.
	\end{enumerate}
\end{lemma}
\begin{proof}
	The tensor product over $\CC$ is exact. So for any simple $\lieg$-module $S$ the sequnce
	\[
	0\to V'_{-S} \to V_{-S} \to V''_{-S} \to 0
	\]
	is exact.
	It induces the homology long exact sequence
	\[
	\xrightarrow{\partial} H_k(V'_{-S}) \to H_k(V_{-S}) \to H_k(V''_{-S}) \xrightarrow{\partial} H_{k-1}(V'_{-S}) \to.
	\]
	If $H_k(V'_{-S})=H_k(V''_{-S})=0$ for all $k$, then $H_k(V_{-S})$ are also zero. This proves (1).
	
	For (2) observe that the connecting homomorphisms $\partial$ are all zero if the exact sequence splits. So, $H_k(V_{-S}) = H_k(V'_{-S})\oplus H_k(V''_{-S})$ and $\sigma_{\lieg}(V)=\sigma_\lieg(V')\cup \sigma_\lieg(V'')$.
\end{proof}

\subsection{Semisimple Lie algebras}%
\label{sec:semisimple}
Let $\lieg$ be a semisimple Lie algebra of dimension $n$. Note that $\bigwedge^n
\lieg$ is isomorphic to the trivial $\lieg$-module. It follows from Proposition \ref{p:definitions} that the Taylor spectrum of a $\lieg$-module $V$ can be described as the set
\[
\sigma_\lieg(V) = \{ S \in \hat\lieg \colon H^k(\lieg, V_{-S})\neq 0 \text{ for some } k\}.
\]
The next lemma helps us to compute the spectrum for finite-dimensional $\lieg$-modules. 
\begin{lemma} \label{l:cohomologyofsemisimple}
	Let $S$ be a simple $\lieg$-module. Assume that $S\not\cong \CC$. Then 
	\[
	H^k(\lieg, S) = 0 \text{ for all } k.
	\]
\end{lemma}
\begin{proof}
	\cite[Theorem 7.8.9]{weibel}
\end{proof}

If the module is trivial, then, obviously, $H^0(\lieg, \CC) \cong \CC$. The following theorem gives a complete description of the spectrum for semisimple Lie algebras.
\begin{theorem} \label{t:ssspectrumissubmodules}
	Let $V$ be a finite-dimensional $\lieg$-module. The Taylor spectrum of $V$ coincides with the set of
	simple components of $V$:
	\begin{equation} \label{e:specforss}
		\sigma_\lieg(V) = \{S \in \hat\lieg \colon S \text{ is isomorphic to a submodule of } V\}.
	\end{equation}
\end{theorem}
\begin{proof}
	By Lemma \ref{l:exactsequence}(2), the spectrum of $V$ is the union of spectra of its simple submodules, so we may assume that $V$ is simple. Let $S$ be a simple $\lieg$-module. Let $V_{-S}=\oplus_{i=1}^m S_i$ be a simple decomposition of $V_{-S}$. Then we have $H^k(\lieg, V_{-S}) \cong \oplus_{i=1}^m H^k(\lieg, S_i)$. By Lemma \ref{l:cohomologyofsemisimple}, $H^k(\lieg, V_{-S})$ is nonzero for some $k$ if and only if $S_i\cong \CC$ for some $i$. On the other hand, the module $V_{-S}$ is isomorphic to $\Hom_\CC(S, V)$ and
	$(V_{-S})^\lieg\cong\Hom_\lieg(S, V)$, so it is one-dimensional if $S\cong V$ and zero
	otherwise. It follows that $S$ lies in the spectrum of $V$ if and only if it is isomorphic to a submodule of $V$.
\end{proof}

In fact, formula \eqref{e:specforss} holds even for Banach $\lieg$-modules. We need the following fact to prove this.
\begin{proposition}
	Any Banach module $V$ over a semisimple Lie algebra $\lieg$ is the union of its
	finite-dimensional submodules. 
\end{proposition}
\begin{proof}
	 \cite[Corollary 5, \S30]{beltita}
\end{proof}
In other words, $V$ is the colimit of the filtered diagram $\mathfrak{V}$ of its finite-dimensional submodules. Recall that the homology and the tensor product commute with filtered
colimits. Thus, we can strengthen Theorem \ref{t:ssspectrumissubmodules}.
\begin{corollary} \label{cor:banachss}
	The Taylor spectrum of a Banach $\lieg$-module $V$ coincides with the set of isomorphism classes of simple submodules of $V$. 
\end{corollary}
\begin{proof}
	The filtered diagram $\mathfrak{V}$ of modules induces the filtered diagram $H_*(\lieg, \mathfrak{V}_{-S})$ of vector spaces. Moreover, $H_*(\lieg, V_{-S})$ is the colimit of $H_*(\lieg, \mathfrak{V}_{-S})$. Since all the maps in $\mathfrak{V}$ are inclusions, so are the maps in $H_*(\lieg, \mathfrak{V}_{-S})$. Thus, $H_*(\lieg, V_{-S})$ is nonzero if and only if $S \in \sigma(W)$ for some $W\subset V$. This means that $S \in \sigma(V)$ if and only if $S$ is isomorphic to a submodule of $V$.
\end{proof}

\subsection{Extensions of Lie algebras}
Let $\lieh$ be an arbitrary Lie algebra and $\lambda$ be a character of $\lieh$. By a \textit{one-dimensional extension} of $\lieh$ we mean an exact
sequence of Lie algebras
\[
0 \to \CC_\lambda\rightarrow \lieg \xrightarrow\pi \lieh \to 0.
\]
Here $\CC_\lambda$ is an ideal of $\lieg$ and the commutator in $\lieg$ satisfies $[g,c] = \lambda(\pi(g))c$ for all $g \in \lieg$ and $c \in \CC_\lambda$.  The fundamental result of the Lie algebra cohomology theory states that isomorphism classes of such extensions are in one-to-one
correspondence with the set $H^2(\lieh,\CC_\lambda) = \Ext^2_{U\lieh}(\cmpl, \cmpl_\lambda)$. We recall how to construct the bijection. Let $\xi \in Hom_\CC(\bigwedge^2\lieh, \CC_\lambda)$ be a
cocycle. We define a Lie bracket on the vector space $\lieg =   \lieh \oplus \CC_\lambda$ by 
\[
[(h_1, c_1), (h_2, c_2)]_\xi= ([h_1, h_2], \lambda(h_1) c_2 - \lambda(h_2) c_1 +
\xi(h_1\wedge h_2)), 
\]
for all $h_i \in \lieh,~c_i\in \CC_\lambda$. It can be shown that this formula satisfies the axioms of a Lie bracket and that cohomologous cocycles induce isomorphic extensions (see \cite[Theorem 7.6.3]{weibel}).

We adapt the techniques for computing cohomology of Lie algebra extensions from \cite[Chapter II.\S6]{guichardet}. From now on and till the end of the section we write $\lieg$ for a one-dimensional extension
of $\lieh$, represented by a cocycle $\xi \in Hom_\CC(\bigwedge^2\lieh, \CC_\lambda)$. We also use
the notation $\mathfrak{c}$ for the ideal $\CC_\lambda\subset \lieg$. Let $V$ be a $\lieh$-module.
We denote by $V^\pi$ the $\lieg$-module which is obtained from $V$ via the homomorphism $\pi$. 
Note that if $S \in \hat\lieh$ is an ireducible $\lieh$-module, then $S^\pi$ is also irreducible as a $\lieg$-module. We use this to identify $\hat\lieh$ with the subset of $\hat\lieg$. 

Assume that we have an $\lieh$-module $V$ and we want to compute the spectrum $\sigma_\lieg (V^\pi)$. The following proposition gives us some clues.
\begin{proposition} \label{p:1ext}
	Let $\lieg$ and $V$ be as above. Then 
	\begin{enumerate}
		\item $\sigma_\lieg(V^\pi) \subset \hat\lieh \subset \hat\lieg$;
		
		%\item For any $S \in \hat\lieh$, the Hochshild-Serre spectral sequence induces linear maps $d_k^* \colon H_k(\lieh, V_{-S})
		%\to H_{k - 2}(\lieh,  V_{-S}\otimes_\CC \CC_\lambda)$ for $0 \leq k \leq n+2$. Moreover, $S$ lies in 
		%$\sigma_\lieg(V^\pi)$ if and only if $d^*_k$ is not an isomoprhism for some $k$;
		
		\item If $S\in \sigma_\lieg(V^\pi)$, then either $S \in \sigma_\lieh(V)$ or
		$S_{-\lambda}\in \sigma_\lieh(V)$;
		
		\item If the extension is central, i.e., $\lambda=0$, then $\sigma_\lieg(V^\pi) = \sigma_\lieh(V)$.
	\end{enumerate}
\end{proposition}
\begin{proof}
	Let $S \in \hat\lieg \setminus \hat\lieh$. It is easy to verify that $\mathfrak{c}S=\{c\cdot s
	\colon s \in S,~ c \in \mathfrak{c}\}$ is a submodule of $S$. Since $S$ is irreducible, it is
	either a zero submodule or the whole $S$. If it is zero, then $S$ is actually an $\lieh$-module and it contradicts
	our assumption. Thus, $\mathfrak{c}S = S$ and $S_\mathfrak{c} = 0$. The same argument can be used
	to show that $S^\mathfrak{c}=0$. 
	
	The Hochschild-Serre spectral sequence
	$E^2_{p, q} = H_p(\lieh, H_q(\mathfrak{c}, V_{-S}))$ converges to $H_k(\lieg, V_{-S})$ (see \cite[7.5]{weibel}). It
	suffices to prove that $H_q(\mathfrak{c}, V_{- S})= 0$ for all $q\geq 0$. The only possible nonzero homology
	groups are $H_0(\mathfrak{c}, V_{-S}) = (V_{-S})_\mathfrak{c} \cong (S^*)_\mathfrak{c}.
	\otimes_\CC V = 0$ and $H_1(\mathfrak{c}, V_{-S}) = (V_{-S})^\mathfrak{c} \cong
	(S^*)^\mathfrak{c} \otimes_\CC V = 0$ since $V$ is $\mathfrak{c}$-invariant. Therefore $S$ is not in $\sigma_\lieg(V^\pi)$ by Proposition \ref{p:definitions}. We have proved (1).
	
	Suppose now that $S \in \hat\lieh$. Then the $\lieh$-modules $H_0(\mathfrak{c},
	V_{-S})$ and $H_1(\mathfrak{c}, V_{-S})$ are isomorphic to $V_{-S}$ and $V\otimes_\CC
	\CC_\lambda$ respectively. Since $H_k(\mathfrak{c}, V_{-S})=0$ for $k>1$, the Hochschild-Serre spectral sequence stabilizes on the third page. The differntials in $E^2$ are $d_k\colon H_k(\lieh, V_{-S}) \to H_{k - 2}(\lieh,
	V_{-S}\otimes_\CC \CC_\lambda)$. The spectral sequence collapses if and only
	if all the differentials are isomorphisms of vector spaces. 
	
	If $d_k$ is not an isomorphism for some $k$, then either
	$H_{k - 2}(\lieh,  V_{-S}\otimes_\CC \CC_\lambda)$ or $H_k(\lieh, V_{-S})$ are nonzero. This means that $S$ or $S_{-\lambda}$ are in $\sigma_\lieh(V)$. This proves (2).
	
	Suppose that $\lambda=0$. Since $S\cong S_{-\lambda}$, we have $\sigma_{\lieg}(V^\pi)\subset \sigma_\lieh(V)$ by (2). Let $S$ be in $\sigma_\lieh(V)$. There is an integer $k\geq0$ such that $H^k(\lieh,V_{-S})\neq0$ and $H^i(\lieh, V_{-S})=0$ for $i<k$. The differential $d_k\colon H_k(\lieh, V_{-S}) \to H_{k - 2}(\lieh,V_{-S})$ is not an isomorphism, so $H_k(\lieg, V_{-S})\neq 0$. This proves (3).
\end{proof}

\section{Case of solvable Lie algebras}%
\label{sec:solvable}
\subsection{Trivial module}
In this section $\lieg$ denotes an arbitrary solvable Lie algebra of dimension $n$. By Lie's
theorem, every simple $\lieg$-module is one-dimensional, so we identify $\hat\lieg$ with the
space $(\lieg/[\lieg,\lieg])^*$ of characters. 

Let $V$ be a finite-dimensional $\lieg$-module. By \emph{weights} $\omega(V) \subset \hat\lieg$ we mean the set of diagonal matrix entries in a triangular
basis for $V$. It is independent on the choice of such basis and can be also described
as the set of one-dimensional subquotients of $V$ (Jordan-H\"older theorem). Consider the adjoint
representation $\ad\lieg \in \lmod\lieg$. The weights $\omega(\ad\lieg)$ are called
\emph{Jordan-H{\"o}lder values} of $\lieg$.  We denote by $2\rho$ the sum of all Jordan-H{\"o}lder values
with multiplicites. If $(g_1,...,g_n)$ is a triangular basis of $\ad\lieg$, then we have 
\[
g \cdot g_1\wedge \cdots \wedge g_n = 2\rho(g) g_1\wedge \cdots \wedge g_n
\]
for $g\in\lieg$. So, $\bigwedge^n\lieg$ is a one-dimensional $\lieg$-module with character $2\rho$.
\begin{theorem} \label{t:jordanholder}
	Let  $\lieg$ be a solvable Lie algebra of dimension $n$. If $\lambda \in \hat\lieg$ is in the
	spectrum $\sigma_\lieg(\CC)$, then it is the sum of at most $n$ Jordan-H{\"o}lder values of
	$\lieg$. Moreover, if $\lambda \in \sigma_\lieg(\CC)$, then $2\rho - \lambda$ is also in
	$\sigma_\lieg(\CC)$.
\end{theorem}
\begin{proof}
	The proof is by induction on $n$. The theorem is trivial for $n=1$. Assume that the statement holds
	for all solvable Lie algebras of dimension $n-1$. Choose any one-dimensional ideal
	$\mathfrak{c}$ in $\lieg$ and denote the corresponding character by $\mu$. By Proposition \ref{p:1ext}.(2) we know that the spectrum $\sigma_\lieg(\CC)$ is the subset of $\{0, \mu\} + \sigma_{\lieg/\mathfrak{c}}(\CC)$. By induction, any $\nu \in
	\sigma_{\lieg/\mathfrak{c}}(\CC)$ is the sum of at most $n-1$ Jordan-H{\"o}lder values of
	$\lieg/\mathfrak{c}$, which are also Jordan-H{\"o}lder values of $\lieg$. Since $\mu$ is a Jordan-H{\"o}lder value of $\lieg$, we conclude that any element of the spectrum $\sigma_\lieg(\CC)$ is the sum of at most $n$ Jordan-H{\"o}lder values of $\lieg$.
	
	For the second assertion we use Proposition \ref{p:definitions}. The trivial module is isomorphic to
	its dual, so if $\CC_\lambda$ is in the spectrum of $\CC$, then $\CC_\lambda^* \otimes_\CC
	\bigwedge^n\lieg \cong \CC_{2\rho - \lambda}$ is also in the spectrum.
\end{proof}
\begin{remark}
	 An equivalent statement was first obtained by Millionschikov in \cite[Theorem 3.1]{mill}. Theorem \ref{t:jordanholder} is a slight reformulation of that result in terms of the Taylor spectrum.
\end{remark}

Obviously, $0$ is always in the spectrum and, hence, so is $2\rho$. So,
generally, there are often more than one element in the spectrum of $\CC$. 
\begin{example}
	Let $\lieg$ be the 3-dimensional solvable Lie algebra with basis $e_1, e_2, e_3$ and the
	commutator given by $[e_1, e_2] = e_2$ and $[e_1, e_3] = \lambda e_3$ for some $\lambda \in
	\CC$. The space of characters is one-dimensional, so we identify it with $\CC$ by the evaluation at
	$e_1$. Then $\sigma_\lieg(\CC) = \{0,1,\lambda, 1 + \lambda\}$. Indeed, $0$ and $2\rho = 1 +
	\lambda$ are always in the spectrum and for $1$ and $\lambda$ the first homology groups are
	non-vanishing.
\end{example}

\subsection{Nilpotent Lie algebras}
The following well-known fact in representation theory for nilpotent Lie algebras is crucial for the computation of the Taylor spectrum. 
\begin{lemma} \label{l:nilpdecomp}
	Let $V$ be a finite-dimensional indecomposable module over nilpotent Lie algebra $\lieg$. Then the set of weights $\omega(V)$ consists of one element.
\end{lemma}
\begin{proof}
	\cite[Chapter VII, Proposition 9]{bourbaki}
\end{proof}
We call modules with only one weight \textit{monoweighted}. Lemma \ref{l:nilpdecomp} shows us that any
finite-dimensional module over a nilpotent Lie algebra can be decomposed in the sum of monoweighted
submodules.  We are ready to formulate the main result.
\begin{theorem} \label{t:nilpspectrum}
	Let $\lieg$ be a nilpotent Lie algebra and $V$ be a finite-dimensional $\lieg$-module. Then the
	spectrum $\sigma_\lieg(V)$ coincides with the set of isomorphism classes of one-dimensional submodules of $V$.
\end{theorem}
\begin{proof}
	The spectrum of $V$ is the union of spectra of its indecomposable submodules by Lemma \ref{l:exactsequence}(2), so we may assume that $V$ is indecomposable. Then by Lemma \ref{l:nilpdecomp}, $V$ is monoweighted. If $\mu$ is the only weight of $V$, then $\omega(V_{-\mu})=\{0\}$ and $\sigma_\lieg(V_{-\mu})=\sigma_\lieg(V)-\mu$, so we may additionally assume that $\omega(V)=\{0\}$. The proof is by induction on the dimension of $V$.
	
	By Engel's theorem, all Jordan-H{\"o}lder values of $\lieg$
	are zero. Using Theorem \ref{t:jordanholder} we conclude that the assertion holds for 
	$V\cong \CC$. If $\lambda$ is a character of $\lieg$, then $\sigma_\lieg(\CC_\lambda)=\sigma_\lieg(\CC)+\lambda$. This means that the assertion holds for one-dimensional modules. 
	
	Suppose now that $\dim V = m$ and that the theorem holds for modules of dimension less than $m$. Choose a one-dimensional trivial submodule of $V$.
	We have a short exact sequence of modules
	\[
	0\to \CC \to V \to V/\CC \to 0.
	\]
	By Lemma \ref{l:exactsequence}(1), $\sigma(V) \subset
	\sigma(\CC) \cup \sigma(V/\CC)$. But $\sigma(\CC) = \sigma(V/\CC) = \{0\}$ by induction. On the other hand,
	$0$ is in the spectrum of $V$ since $H_0(\lieg, V) = V_\lieg \neq 0$. This completes the
	proof.
\end{proof}

\subsection{Borel subalgebras of semisimple Lie algebras}
Here we use the terminology from the theory of semisimple Lie algebras. We refer the reader to \cite[Chapter 2]{humphreys} for details. 

For the rest of the section $\lies$ is a semisimple Lie algebra and $\lieg$ is a Borel subalgebra \mbox{of $\lies$}. It is known
that $\lieg$ is isomorphic to a semidirect product $\lieh \ltimes \lien$, where
$\lieh$ is a Cartan subalgebra of $\lies$ and $\lien = [\lieg, \lieg]$.  We denote the root system of $\lies$ relative to $\lieh$ by $\Delta \subset \lieh^*
= \hat\lieg$. We write $\Delta^+ \subset \Delta$ for the subset of positive roots. Elements of $\Delta^+$
are simply nonzero Jordan-H\"older values of $\lieg$. We use the notation $W(\Delta)$ for the Weyl group. For any
element $w\in W$ its length is denoted by $l(w)$. 

The aim of the subsection is to describe the spectrum of irreducible $\lies$-modules considered
as $\lieg$-modules. First we describe the cohomology of such modules over $\lien$ by Kostant's theorem. Then we use the formula for the cohomology of a semidirect product to compute the cohomology over $\lieg$.

Let $V$ be a $\lies$-module of highest weight $\lambda$. The cohomology groups $H^k(\lien, V)$ are naturally $\lieh$-modules after a canonical identification $\lieh=\lieg/\lien$. Kostant's theorem gives an explicit formula for these modules. 
\begin{lemma}[Kostant's theorem] \label{l:kostantthm}
	As an $\lieh$-module,
	$H^k(\lien, V)$ is the sum of one-dimensional modules with weights
	\[
	w(\lambda + \rho) - \rho,~w\in W(\Delta), ~l(w) = k,
	\]
	where $\rho$ is the half-sum of positive roots (or, equivalently, Jordan-H\"older values).
\end{lemma}
\begin{proof}
	 \cite[Theorem 6.12]{knapp}
\end{proof}
We say that an abelian Lie algebra $\lieh$ acts \textit{torally} on an $\lieh$-module $W$ if $W$ is a direct sum
of one-dimensional submodules. If $\lieh$ is a Cartan subalgebra of a semisiple Lie algebra
$\lies$, then  $\lieh$ acts torally on any finite dimensional $\lies$-module. If $\lieg = \lieh \ltimes
\lien$ is the Borel subalgebra of $\lies$, then $\lieh$ also acts torally on $\lien$. 
\begin{lemma} \label{l:semiproduct}
	For a finite-dimensional $\lies$-module $V$ we have
	\[
		H^k(\lieg, V) = \bigoplus_{p + q = k} \bigwedge^p\lieh^* \otimes_\CC H^q(\lien, V)^\lieh
	\]
\end{lemma}
\begin{proof}
	It is a special case of \cite[Theorem 4]{coll}. 
\end{proof}

The following theorem gives a complete description of the Taylor spectrum for Borel subalgebras in the special case of modules restricted from a semisimple algebra.
\begin{theorem} \label{t:borelspectrum}
	Let $\lies$ be a semisimple Lie algebra and $V$ be a simple $\lies$-module of highest weight $\lambda$. Let $\lieg$ be a Borel subalgebra of $\lies$. Then the Taylor spectrum of $V$ considered as a $\lieg$-module is given by
	\begin{equation} \label{e:borelspec}
	\sigma_\lieg(V) = \{ \rho + w(\lambda + \rho)\colon w \in W(\Delta)\},
	\end{equation}
	where $\rho$ is the half-sum of positive roots and $W(\Delta)$ is the Weyl group.
\end{theorem}
\begin{proof}
	By Proposition \ref{p:definitions}, $\nu \in \sigma_{\lieg}(V)$ if and only if $H^k(\lieg,
	\bigwedge^n\lieg \otimes_\CC V_{-\nu})$ is nonzero for some $k$. The module $\bigwedge^n \lieg$ is isomorphic to $\CC_{2\rho}$, so $\bigwedge^n\lieg
	\otimes_\CC V_{-\nu} \cong V_{2\rho - \nu}$. By Lemma \ref{l:semiproduct}, we have
	\[
	H^k(\lieg, V_{2\rho-\nu}) = \bigoplus_{p + q = k} \bigwedge^p\lieh^* \otimes_\CC H^q(\lien, V_{2\rho-\nu})^\lieh = \bigoplus_{p + q = k} \bigwedge^p\lieh^* \otimes_\CC (H^q(\lien, V)_{2\rho-\nu})^\lieh.
	\]
	It follows that $\nu\in \sigma_\lieg(V)$ if and only if $(H^q(\lien, V)_{2\rho-\nu})^\lieh$ is nonzero for some $q$. By Lemma \ref{l:kostantthm}, $H^q(\lien, V)_{2\rho-\nu}$ is the sum of one-dimensional modules with weights $w(\lambda+\rho) +\rho -\nu$ with $w\in W$ of length $l(w)=q$. We conclude that $\nu \in \sigma_\lieg(V)$ if and only if $\nu = w(\lambda+\rho)+\rho$ for some $w\in W(\Delta)$.
\end{proof}
\begin{remark}
	A formula equivalent to \eqref{e:borelspec} was obtained for the trivial module in \cite[Section 4.4]{total}. Theorem \ref{t:borelspectrum} generalizes this to arbitrary $\lies$-modules and gives the description in terms of the Taylor spectrum.
\end{remark}

\end{document}